\documentclass[a4paper,reqno]{amsart}
\usepackage{amssymb, amsmath, amscd}
\usepackage{color}

\numberwithin{equation}{section}
\newtheorem{thm}{Theorem}

\begin{document}

\title{\bf Correction to ``minimal unit vector fields"}
\author{S. H. Chun, J.  H. Park, and K. Sekigawa}
\address{SHC, JHP:
Department of Mathematics, Sungkyunkwan University,
 Suwon 440-746, Korea. Email: cshyang@skku.edu and parkj@skku.edu.}
 \address{KS: Department of Mathematics, Faculty of Science,
 Niigata University, Niigata, 950-2181, Japan. Email:
 sekigawa@math.sc.niigata-u.ac.jp.}
 \date{}

\maketitle

\footnote{2000 \textit{Mathematics Subject Classification:} 53C25, 53D10}
\footnote{\textit{Keywords:}
minimal vector field}

\setcounter{section}{1}

The paper ``Minimal unit vector fields" by O. Gil-Medrano and E. Llinares-Fuster \cite{GilLli1}.
is a seminal paper in the field that has been cited by many authors -- see, for example,
\cite{B10,F09,H10,P10,P09,Y10} to give just a few of the recent citations. It contains, however,
a minor technical mistake in Theorem 14 that is important to fix.
In this short note, we will provide a correction to
that result. We begin by establishing the requisite notation. Let
$(M,g)$ be a Riemannian manifold Assume that there exists a unit
Killing vector field $V$ on $M$. Let $\nabla$ be the Levi-Civita
connection of $g$. Set:
$$\begin{array}{ll}
L_V=\operatorname{id}+(\nabla V)^t\circ\nabla V,&
f(v):=\sqrt{\det(L_V)},\\
K_V:=f(V)L_V^{-1}\circ(\nabla V)^t,&
\omega_V(X):=\operatorname{Tr}(Z\rightarrow\nabla_ZK_V(X))\,.\vphantom{\vrule height 12pt}
\end{array}$$
Since $V$ is a Killing vector field, the rank of $\nabla V$ must be
even. We further normalize the choice of frame so that $\nabla
V(E_i)=-\lambda_iE_{i^\star}$ and $\nabla V(E_{i^*})=\lambda_iE_i$
for $i\in\{1,...,m\}$ and $\nabla V(E_\alpha)=0$ for
$2m+1\le\alpha\le n$. Thus we can take our frame to be:
$$\{E_1,E_{1^\star},...,E_{m},E_{m^\star},...,E_n=V\}\,.$$
We adopt the sign convention for the curvature given in \cite{GilLli1}, namely:
$$R(x,y,z,w)=-g((\nabla_x\nabla_y-\nabla_y\nabla_x-\nabla_{[x,y]})z,w)\,.$$
 In \cite{GilLli1} the authors stated as Theorem 14 the following result:
\begin{thm}\label{Gilthm}
Let $V$ be a unit Killing vector field,
then $\omega_V=f\tilde\rho_V$, where $\tilde\rho_V(X)$
is defined to be
$$\sum_{j=1}^n({R}((L_V^{-1}\circ\nabla V)(X), (L_V^{-1}\circ\nabla V)(E_j),V,E_j)
+{R}(L_V^{-1}(X), L_V^{-1}(E_j),V,E_j)).$$
Consequently, $V$ is minimal if and only if the 1-form $\tilde\rho_V$ annihilates
$\mathcal{H}^V$.
\end{thm}
Unfortunately, this is not quite correct. The correct result is as follows:
\begin{thm}\label{corr}
Let $V$ be a unit Killing vector field,
then $$\omega_V(X)=f\tilde\rho_V(X) -((L_V^{-1}\circ\nabla V)X)f\,,$$ where $\tilde\rho_V(X)$
is defined to be
$$\sum_{j=1}^n(R((L_V^{-1}\circ\nabla V)(X), (L_V^{-1}\circ\nabla V)(E_j),V,E_j)
{-R}(L_V^{-1}(X), L_V^{-1}(E_j),V,E_j)).$$
Consequently, $V$ is minimal if and only if $f\tilde\rho_V(X) =((L_V^{-1}\circ\nabla V)X)f$
for any vector field $X$ orthogonal to $V$.
\end{thm}
\begin{proof} From the proof of Theorem 14 \cite{GilLli1}, we obtain
\begin{equation}\label{e_i}
\begin{split}
\frac{1}{f} \omega_V(E_i)=
&\frac{2\lambda_i}{1+\lambda_i^2}\sum_{j=1}^m\frac{\lambda_j}{1+\lambda_j^2}E_{i^*}(\lambda_j)
+\frac{1-\lambda_i^2}{(1+\lambda_i^2)^2}E_{i^*}(\lambda_i)\\
&+\frac{\lambda_i}{1+\lambda_i^2}\sum_{j=1}^{n-1}G_{ji^*}^j
+\sum_{j=1}^m \frac{\lambda_j}{1+\lambda_j^2}(G_{ji}^{j^*}-G_{j^*i}^{j}),
\end{split}
\end{equation}
\begin{equation}
\begin{split}\label{e_i*}
\frac{1}{f} \omega_V(E_{i^*})=
&-\frac{2\lambda_i}{1+\lambda_i^2}\sum_{j=1}^m\frac{\lambda_j}{1+\lambda_j^2}E_{i}(\lambda_j)
-\frac{1-\lambda_i^2}{(1+\lambda_i^2)^2}E_{i}(\lambda_i)\\
&-\frac{\lambda_i}{1+\lambda_i^2}\sum_{j=1}^{n-1}G_{ji}^j
-\sum_{j=1}^m
\frac{\lambda_j}{1+\lambda_j^2}(G_{j^*i^*}^{j}-G_{ji^*}^{j^*}),
\end{split}
\end{equation}
\begin{equation}\label{e_alpha}
\frac{1}{f} \omega_V(E_{\alpha})=
-\sum_{j=1}^m \frac{\lambda_j}{1+\lambda_j^2}(G_{j^*\alpha}^{j}-G_{j\alpha}^{j^*}).
\end{equation}
On the other hand, from the definition of $\tilde\rho_V$ in Theorem \ref{Gilthm}, we have
\begin{equation}\label{rho i}
\begin{split}
\tilde\rho_V (E_i)
=&-\frac1{1+\lambda_i^2}\sum_{j=1}^m \frac{1}{1+\lambda_j^2}({R}(E_i, E_j, E_j, V){+R}(E_i, E_{j^*}, E_{j^*},V))\\
&+\frac{\lambda_i}{1+\lambda_i^2}\sum_{j=1}^m \frac{\lambda_j}{1+\lambda_j^2}{R}(E_{j^*}, E_j, E_{i^*}, V)\\
&-\frac1{1+\lambda_i^2}\sum_{\beta=2m+1}^n {R}(E_i, E_\beta, E_\beta, V),
\end{split}
\end{equation}
\begin{equation}\label{rho i*}
\begin{split}
\tilde\rho_V (E_{i^*})
=&-\frac1{1+\lambda_i^2}\sum_{j=1}^m \frac{1}{1+\lambda_j^2}(R(E_{i^*}, E_j, E_j, V)+R(E_{i^*}, E_{j^*}, E_{j^*},V))\\
&-\frac{\lambda_i}{1+\lambda_i^2}\sum_{j=1}^m \frac{\lambda_j}{1+\lambda_j^2}{R}(E_{j^*}, E_j, E_i, V)\\
&-\frac1{1+\lambda_i^2}\sum_{\beta=2m+1}^n R(E_{i^*}, E_\beta, E_\beta, V),
\end{split}
\end{equation}
\begin{equation}\label{rho alpha}
\begin{split}
\tilde\rho_V (E_{\alpha})
=&-\sum_{j=1}^m \frac{1}{1+\lambda_j^2}({R}(E_{\alpha}, E_j, E_j, V){+R}(E_{\alpha}, E_{j^*}, E_{j^*},V))\\
&-\sum_{\beta=2m+1}^n {R}(E_{\alpha}, E_\beta, E_\beta, V).
\end{split}
\end{equation}
Let $G_{ij}^k=g(\nabla_{E_i}E_j,E_k)$ describe the components of covariant differentiation on this frame field. We then have
$$(\nabla V)_i^j=G_{in}^j\quad\text{and}\quad G_{ij}^k=-G_{ik}^j\,.$$
From Lemma 12 in \cite{GilLli1}, for a unit Killing vector field $V$, the components of the
curvature tensor are given by
\begin{equation}\label{Rijkn}
\begin{split}
R_{jikn}=& -E_i ((\nabla V)_j^k)+E_j ((\nabla V)_i^k) \\
&+\sum_{l=1}^{n-1}\{-G_{il}^k(\nabla V)_j^l +G_{jl}^k (\nabla V)_i^l
+G_{ij}^l (\nabla V)_l^k -G_{ji}^l (\nabla V)_l^k \}.
\end{split}
\end{equation}
Using \eqref{rho i} and applying \eqref{Rijkn}, we obtain
\begin{equation}\nonumber
\begin{split}
\tilde\rho_V (E_i)
=&-\frac{1}{1+\lambda_i^2}\sum_{j=1}^m \frac{1}{1+\lambda_j^2}\{\lambda_i G_{ji^*}^j -\lambda_j G_{ij^*}^j
+\lambda_j G_{ji}^{j^*} -\lambda_j G_{ij}^{j^*}\\
&\qquad\qquad\qquad +E_{j^*}(\lambda_j g_{ij}) +\lambda_i G_{j^* i^*}^{j^*}
 +\lambda_j G_{ij}^{j^*} -\lambda_j G_{j^* i}^j +\lambda_j G_{ij^*}^j\}\\
&+\frac{\lambda_i}{1+\lambda_i^2}\sum_{j=1}^m \frac{\lambda_j}{1+\lambda_j^2}\{-E_{j^*}(\lambda_i g_{ij})
-\lambda_j G_{jj}^{i^*} -\lambda_j G_{j^* j^*}^{i^*} -\lambda_i G_{jj^*}^i +\lambda_i G_{j^* j}^i\}\\
&-\frac{1}{1+\lambda_i^2} \sum_{\beta=2m+1}^{n} \lambda_i G_{\beta i^*}^\beta\,.
\end{split}\end{equation}
This yields:
\begin{equation}\label{rho i-2}
\begin{split}
\tilde\rho_V (E_i)
=&\frac{\lambda_i}{1+\lambda_i^2}\sum_{j=1}^m \frac{1}{1+\lambda_j^2}\{(-1+\lambda_j^2)G_{ji^*}^j +(-1+\lambda_j^2)G_{j^* i^*}^{j^*}\}\\
&+\frac{1}{1+\lambda_i^2}\sum_{j=1}^m \frac{\lambda_j}{1+\lambda_j^2}\{(-1+\lambda_i^2)G_{ji}^{j^*} +(1-\lambda_i^2)G_{j^* i}^{j}\}\\
&-\frac{\lambda_i }{1+\lambda_i^2} \sum_{\beta=2m+1}^{n} G_{\beta i^*}^\beta
-\frac{1}{1+\lambda_i^2} E_{i^*}(\lambda_i)
\end{split}
\end{equation}
Similarly, from \eqref{Rijkn}, and \eqref{rho i*}, we obtain
\begin{equation}\nonumber
\begin{split}
\tilde\rho_V (E_i^*)
=&-\frac{1}{1+\lambda_i^2}\sum_{j=1}^m \frac{1}{1+\lambda_j^2}\{-E_{j}(\lambda_i g_{ij})-\lambda_i G_{ji}^j -\lambda_j G_{i^* j^*}^j
+\lambda_j G_{ji^*}^{j^*}\\
&\qquad\qquad\qquad -\lambda_j G_{i^* j}^{j^*} -\lambda_i G_{j^* i}^{j^*}
 +\lambda_j G_{i^* j}^{j^*} -\lambda_j G_{j^* i^*}^j +\lambda_j G_{i^* j^*}^j\}\\
&-\frac{\lambda_i}{1+\lambda_i^2}\sum_{j=1}^m \frac{\lambda_j}{1+\lambda_j^2}\{-E_{j}(\lambda_j g_{ij})
-\lambda_j G_{jj}^{i} -\lambda_j G_{j^* j^*}^{i} +\lambda_i G_{jj^*}^{i^*} -\lambda_i G_{j^* j}^{i^*}\}\\
&+\frac{1}{1+\lambda_i^2} \sum_{\beta=2m+1}^{n} \lambda_i G_{\beta i}^\beta\,.
\end{split}
\end{equation}
This simplifies to become
\begin{equation}\label{rho i*-2}
\begin{split}
\tilde\rho_V (E_i^*)
=&-\frac{\lambda_i}{1+\lambda_i^2}\sum_{j=1}^m \frac{1}{1+\lambda_j^2}\{(-1+\lambda_j^2)G_{ji}^j +(-1+\lambda_j^2)G_{j^* i}^{j^*}\}\\
&+\frac{1}{1+\lambda_i^2}\sum_{j=1}^m \frac{\lambda_j}{1+\lambda_j^2}\{(-1+\lambda_i^2)G_{ji^*}^{j^*} +(1-\lambda_i^2)G_{j^* i^*}^{j}\}\\
&+\frac{\lambda_i }{1+\lambda_i^2} \sum_{\beta=2m+1}^{n} G_{\beta i}^\beta
+\frac{1}{1+\lambda_i^2} E_{i}(\lambda_i)\,.
\end{split}
\end{equation}
From \eqref{rho alpha}, we also have that:
\begin{eqnarray}\tilde\rho_V (E_\alpha)
&=&-\sum_{j=1}^m \frac{1}{1+\lambda_j^2}\{-\lambda_j G_{\alpha j^*}^{j}+\lambda_j G_{j\alpha}^{j^*}-\lambda_j G_{\alpha j}^{j^*}
+\lambda_j G_{\alpha j}^{j^*}-\lambda_j G_{j^* \alpha}^{j}+\lambda_j G_{\alpha j^*}^{j}\}\nonumber\\
&=&\sum_{j=1}^m \frac{\lambda_j}{1+\lambda_j^2} (G_{j^* \alpha}^j -G_{j\alpha}^{j^*})
\label{rho alpha-2}\end{eqnarray}
Comparing \eqref{e_i} $\sim$ \eqref{e_alpha} and \eqref{rho i-2} $\sim$ \eqref{rho alpha-2},
we see that \eqref{e_i} is not equal to \eqref{rho i-2} and \eqref{e_i*} is not equal to \eqref{rho i*-2}.
Also \eqref{e_alpha} is not equal to \eqref{rho alpha-2}, that is, it is impossible to obtain Theorem \ref{Gilthm}.

For this reason we shall, instead, use the definition of $\tilde\rho_V (X)$ which is given
in Theorem \ref{corr}. Then we have the relation
\begin{equation}\label{rho i-1}
\begin{split}
\tilde\rho_V (E_i)
=&\frac1{1+\lambda_i^2}\sum_{j=1}^m \frac{1}{1+\lambda_j^2}({R}(E_i, E_j, E_j, V){+R}(E_i, E_{j^*}, E_{j^*},V))\\
&+\frac{\lambda_i}{1+\lambda_i^2}\sum_{j=1}^m \frac{\lambda_j}{1+\lambda_j^2}{R}(E_{j^*}, E_j, E_{i^*}, V)\\
&+\frac1{1+\lambda_i^2}\sum_{\beta=2m+1}^n {R}(E_i, E_\beta, E_\beta, V),
\end{split}
\end{equation}
the relation
\begin{equation}\label{rho i*-1}
\begin{split}
\tilde\rho_V (E_{i^*})
=&\frac1{1+\lambda_i^2}\sum_{j=1}^m \frac{1}{1+\lambda_j^2}(R(E_{i^*}, E_j, E_j, V)+R(E_{i^*}, E_{j^*}, E_{j^*},V))\\
&-\frac{\lambda_i}{1+\lambda_i^2}\sum_{j=1}^m \frac{\lambda_j}{1+\lambda_j^2}{R}(E_{j^*}, E_j, E_i, V)\\
&+\frac1{1+\lambda_i^2}\sum_{\beta=2m+1}^n R(E_{i^*}, E_\beta, E_\beta, V),
\end{split}
\end{equation}
and the relation
\begin{equation}\label{rho alpha-1}
\begin{split}
\tilde\rho_V (E_{\alpha})
=&\sum_{j=1}^m \frac{1}{1+\lambda_j^2}({R}(E_{\alpha}, E_j, E_j, V){+R}(E_{\alpha}, E_{j^*}, E_{j^*},V))\\
&+\sum_{\beta=2m+1}^n {R}(E_{\alpha}, E_\beta, E_\beta, V).
\end{split}
\end{equation}
Using the Lemma 12 \cite{GilLli1} and applying Equations (\ref{rho
i-1}) -- (\ref{rho alpha-1}), we obtain:
\begin{equation}\nonumber
\begin{split}
\tilde\rho_V (E_i)
=&\frac{1}{1+\lambda_i^2}\sum_{j=1}^m \frac{1}{1+\lambda_j^2}\{\lambda_i G_{ji^*}^j -\lambda_j G_{ij^*}^j
+\lambda_j G_{ji}^{j^*} -\lambda_j G_{ij}^{j^*}\\
&\qquad\qquad\qquad +E_{j^*}(\lambda_j g_{ij}) +\lambda_i G_{j^* i^*}^{j^*}
 +\lambda_j G_{ij}^{j^*} -\lambda_j G_{j^* i}^j +\lambda_j G_{ij^*}^j\}\\
&+\frac{\lambda_i}{1+\lambda_i^2}\sum_{j=1}^m \frac{\lambda_j}{1+\lambda_j^2}\{-E_{j^*}(\lambda_i g_{ij})
-\lambda_j G_{jj}^{i^*} -\lambda_j G_{j^* j^*}^{i^*} -\lambda_i G_{jj^*}^i +\lambda_i G_{j^* j}^i\}\\
&+\frac{1}{1+\lambda_i^2} \sum_{\beta=2m+1}^{n} \lambda_i G_{\beta i^*}^\beta\,.
\end{split}\end{equation}
Consequently
\begin{equation}\nonumber
\begin{split}
\tilde\rho_V (E_i)
=&\frac{\lambda_i}{1+\lambda_i^2}\sum_{j=1}^m \frac{1}{1+\lambda_j^2}\{(1+\lambda_j^2)G_{ji^*}^j +(1+\lambda_j^2)G_{j^* i^*}^{j^*}\}\\
&+\frac{1}{1+\lambda_i^2}\sum_{j=1}^m \frac{\lambda_j}{1+\lambda_j^2}\{(1+\lambda_i^2)G_{ji}^{j^*} -(1+\lambda_i^2)G_{j^* i}^{j}\}\\
&+\frac{\lambda_i }{1+\lambda_i^2} \sum_{\beta=2m+1}^{n} G_{\beta i^*}^\beta
+\frac{1-\lambda_i^2}{(1+\lambda_i^2)^2} E_{i^*}(\lambda_i).
\end{split}
\end{equation}
This yields:
\begin{equation}\label{rho_i}
\begin{split}
\tilde\rho_V (E_i)
=&\frac{\lambda_i}{1+\lambda_i^2}\{\sum_{j=1}^m G_{ji^*}^j +\sum_{j=1}^m G_{j^* i^*}^{j^*} +\sum_{\beta=2m+1}^{n} G_{\beta i^*}^\beta\}\\
&+\sum_{j=1}^m \frac{\lambda_j}{1+\lambda_j^2}(G_{ji}^{j^*}-G_{j^* i}^{j}) +\frac{1-\lambda_i^2}{(1+\lambda_i^2)^2} E_{i^*}(\lambda_i)\\
=&\frac{\lambda_i}{1+\lambda_i^2}\sum_{j=1}^{n-1}G_{ji^*}^j+\sum_{j=1}^m \frac{\lambda_j}{1+\lambda_j^2}(G_{ji}^{j^*}-G_{j^* i}^{j}) +\frac{1-\lambda_i^2}{(1+\lambda_i^2)^2} E_{i^*}(\lambda_i),
\end{split}
\end{equation}
We continue the computation:
\begin{equation}\nonumber
\begin{split}
\tilde\rho_V (E_i^*)
=&\frac{1}{1+\lambda_i^2}\sum_{j=1}^m \frac{1}{1+\lambda_j^2}\{-E_{j}(\lambda_i g_{ij})-\lambda_i G_{ji}^j -\lambda_j G_{i^* j^*}^j
+\lambda_j G_{ji^*}^{j^*}\\
&\qquad\qquad\qquad -\lambda_j G_{i^* j}^{j^*} -\lambda_i G_{j^* i}^{j^*}
 +\lambda_j G_{i^* j}^{j^*} -\lambda_j G_{j^* i^*}^j +\lambda_j G_{i^* j^*}^j\}\\
&-\frac{\lambda_i}{1+\lambda_i^2}\sum_{j=1}^m \frac{\lambda_j}{1+\lambda_j^2}\{-E_{j}(\lambda_j g_{ij})
-\lambda_j G_{jj}^{i} -\lambda_j G_{j^* j^*}^{i} +\lambda_i G_{jj^*}^{i^*} -\lambda_i G_{j^* j}^{i^*}\}\\
&-\frac{1}{1+\lambda_i^2} \sum_{\beta=2m+1}^{n} \lambda_i G_{\beta i}^\beta,
\end{split}\end{equation}
so that:
\begin{equation}\label{rho_i*}
\begin{split}
\tilde\rho_V (E_i^*)
=&-\frac{\lambda_i}{1+\lambda_i^2}\sum_{j=1}^m \frac{1}{1+\lambda_j^2}\{(1+\lambda_j^2)G_{ji}^j +(1+\lambda_j^2)G_{j^* i}^{j^*}\}\\
&+\frac{1}{1+\lambda_i^2}\sum_{j=1}^m \frac{\lambda_j}{1+\lambda_j^2}\{(1+\lambda_i^2)G_{ji^*}^{j^*} -(1+\lambda_i^2)G_{j^* i^*}^{j}\}\\
&-\frac{\lambda_i }{1+\lambda_i^2} \sum_{\beta=2m+1}^{n} G_{\beta i}^\beta
-\frac{1-\lambda_i^2}{(1+\lambda_i^2)^2} E_{i}(\lambda_i)\\
=&-\frac{\lambda_i}{1+\lambda_i^2}\{\sum_{j=1}^m G_{ji}^j +\sum_{j=1}^m G_{j^* i}^{j^*} +\sum_{\beta=2m+1}^{n} G_{\beta i}^\beta\}\\
&+\sum_{j=1}^m \frac{\lambda_j}{1+\lambda_j^2}(G_{ji^*}^{j^*}-G_{j^* i^*}^{j}) -\frac{1-\lambda_i^2}{(1+\lambda_i^2)^2} E_{i}(\lambda_i)\\
=&-\frac{\lambda_i}{1+\lambda_i^2}\sum_{j=1}^{n-1}G_{ji}^j-\sum_{j=1}^m
\frac{\lambda_j}{1+\lambda_j^2}(G_{j^* i^*}^{j}-G_{ji^*}^{j^*})
-\frac{1-\lambda_i^2}{(1+\lambda_i^2)^2} E_{i}(\lambda_i).
\end{split}
\end{equation}
Finally, we have that
\begin{eqnarray}
\tilde\rho_V (E_\alpha)
&=&\sum_{j=1}^m \frac{1}{1+\lambda_j^2}\{-\lambda_j G_{\alpha j^*}^{j}+\lambda_j G_{j\alpha}^{j^*}-\lambda_j G_{\alpha j}^{j^*}
+\lambda_j G_{\alpha j}^{j^*}-\lambda_j G_{j^* \alpha}^{j}+\lambda_j G_{\alpha j^*}^{j}\}\nonumber\\
&=&-\sum_{j=1}^m \frac{\lambda_j}{1+\lambda_j^2} (G_{j^* \alpha}^j -G_{j\alpha}^{j^*}).\label{rho_alpha}
\end{eqnarray}
From Equation (\ref{e_i}) and (\ref{rho_i}), we see that
$$\frac{1}{f}\omega_V(E_i)-\rho_V(E_i) =\frac{2\lambda_i}{1+\lambda_i^2}\sum_{j=1}^m\frac{\lambda_j}{1+\lambda_j^2}E_{i^*}(\lambda_j).$$
Since
\begin{equation}\label{L1}
((L_V^{-1}\circ\nabla V)E_i)f
=-f\frac{2\lambda_i}{1+\lambda_i^2}\sum_{j=1}^m\frac{\lambda_j}{1+\lambda_j^2}E_{i^*}(\lambda_j),
\end{equation}
we have
$$\frac{1}{f}\omega_V(E_i)-\rho_V(E_i)=- \frac{1}{f}((L_V^{-1}\circ\nabla
V)E_i)f.$$ Similarly, since
\begin{eqnarray}\label{L2}
&&((L_V^{-1}\circ\nabla V)E_{i^*})f
=f\frac{2\lambda_i}{1+\lambda_i^2}\sum_{j=1}^m\frac{\lambda_j}{1+\lambda_j^2}E_{i}(\lambda_j)
\text{ and}\\
&&\label{L3}
((L_V^{-1}\circ\nabla V)E_{\alpha})f=0,
\end{eqnarray}
we have the same results for \eqref{e_i*} and \eqref{rho_i*} and for \eqref{e_alpha} and
\eqref{rho_alpha}, respectively. This completes the proof of the Theorem \ref{corr}.
\end{proof}

\section*{Acknowledgments}This work was supported by the National Research
Foundation of Korea (NRF) grant funded by the Korea government
(MEST) (2012-0005282).

\end{document}